\documentclass{amsart}
\usepackage{mathrsfs,amssymb,setspace}
\usepackage{verbatim, dsfont, enumerate, color, stmaryrd, bbold, bm, amsmath,mathabx}

\theoremstyle{plain}
\newtheorem{theorem}{Theorem}[section]
\newtheorem{lemma}[theorem]{Lemma}
\newtheorem{proposition}[theorem]{Proposition}
\newtheorem{corollary}[theorem]{Corollary}

\theoremstyle{definition}

\theoremstyle{remark}
\newtheorem{remark}[theorem]{Remark}

\usepackage[all]{xy}

\newcommand{\cref}[1]{Corollary \ref{#1}}

\newcommand{\nll}{\ll\!\!\!\!\!/\;}

\begin{document}


\title[Branch Groups with Non-Torsion Rigid Kernels]{A Class of Non-Contracting Branch Groups with Non-Torsion Rigid Kernels}
\author[Sagar Saha]{Sagar Saha}
\address{Department of Mathematics,  Indian Institute of Technology Guwahati, Guwahati, India}
\email{sagarsaha@iitg.ac.in}
\author[K. V. Krishna]{K. V. Krishna}
\address{Department of Mathematics, Indian Institute of Technology Guwahati, Guwahati, India}
\email{kvk@iitg.ac.in}


\begin{abstract}
In this work, we provide the first example of an infinite family of branch groups in the class of non-contracting self-similar groups. We show that these groups are very strongly fractal, not regular branch, and of exponential growth. Further, we prove that these groups do not have the congruence subgroup property by explicitly calculating the structure of their rigid kernels. This class of groups is also the first example of branch groups with non-torsion rigid kernels. As a consequence of these results, we also determine the Hausdorff dimension of these groups.       
\end{abstract}



\maketitle

\vspace{.5cm}

\section{Introduction}
The self-similar groups, a class of groups acting on regular rooted trees, have gained prominence since Grigorchuk's pioneering work in the 1980s. Starting from the Grigorchuk group \cite{grigorchuk1980burnside}, various self-similar groups were studied in the literature; e.g., Gupta-Sidki $p$-groups \cite{gupta1983burnside}, Basilica group \cite{grigorchuk2002torsion}, Hanoi towers group \cite{Grigorchuk2006hanoi}, and their generalizations   \cite{bartholdi2001word, Domenico_2022_basilica, Skipper2020}. These groups have become a rich source of examples exhibiting exotic properties like, amenable but not elementary amenable, just infinite, intermediate word growth, finitely generated infinite torsion. In 1997, Grigorchuk introduced the concept of (weakly) branch groups, inspired by some well-structured subgroup formations in many famous examples.  The previously mentioned examples fall into the intersection of self-similar groups and (weakly) branch groups. The self-similar (weakly) branch groups
have received significant attention due to the rich set of tools they offer for studying various aspects like the congruence subgroup problem, Hausdorff dimension, maximal subgroups, commensurability, L-presentation,  and Engel elements  \cite{Bartholdi2003presentation, Bou-Rabee2020, Noce2020engel, Francoeur2020, Noce2021Hausdorff, Skipper2020}. 

The classical congruence subgroup problem for $SL(n;\mathds{Z})$ (see \cite{Bass_1964}) was extended to automorphism groups of regular rooted trees and studied by various authors. A subgroup $G$ of $\text{Aut}(T)$, the automorphism group of a regular rooted tree  $T$,  has the Congruence Subgroup Property (in short, CSP) if every finite index subgroup of $G$ contains a level stabilizer. This can be restated in terms of profinite completions as follows. The full profinite topology on $G$ is obtained by considering all finite index normal subgroups of $G$ as neighborhood basis for the identity element of $G$. The profinite topology obtained by restricting to the level stabilizers of $G$, rather than all finite index normal subgroups, is known as congruence topology on $G$. The profinite completion (denoted by $\widehat{G}$) and the congruence completion (denoted by $\overline{G}$) of $G$ is the completion of $G$ with respect to full profinite topology and congruence topology, respectively. The group $G$ satisfies CSP if the kernel of the natural surjection $\widehat{G} \twoheadrightarrow \overline{G}$, called the congruence kernel, is trivial, i.e., $\widehat{G}$ and $\overline{G}$ are isomorphic. For details, one may refer to \cite{Bartholdi_2012,Garrido_2016, Skipper2020}.   

In the context of branch groups, the rigid stabilizers of levels define another profinite topology, called branch topology. The branch completion $\widetilde{G}$ of $G$ is the completion of $G$ with respect to the branch topology. For a branch group $G$, the congruence kernel divides into two parts called branch kernel and rigid kernel which are the kernels of natural surjections $\widehat{G} \twoheadrightarrow \widetilde{G}$ and $\widetilde{G} \twoheadrightarrow \overline{G}$, respectively. The congruence subgroup problem for branch groups asks whether these surjections are isomorphisms; if not, quantitatively describe the respective kernels. 

There are good number of examples of branch groups with trivial congruence kernels (e.g., see \cite{Garrido_2017,  Garrido_2016_gupta, Garrido_2019, Grigorchuk2000}). The first example of branch groups with non-trivial congruence kernels was given by Pervova \cite{Pervova_2007}, although these groups have trivial rigid kernels. The Hanoi towers group $\Gamma_3$ on three pegs is the first branch group with non-trivial rigid kernel; in fact, the rigid kernel is Klein four-group \cite{Bartholdi_2012}. Bartholdi et al. \cite{Bartholdi_2012} studied the congruence subgroup problem for self-similar regular branch groups, as these groups possess finite index branching subgroups, and established that for such a group the branch kernel is abelian and the rigid kernel has finite exponent. Recently, Garrido and Sunic \cite{rigid2023branch} provided the first example of branch groups with infinite rigid kernels, and determined them as the elementary 2-group  $\prod_{\mathds{N}}\mathds{Z}/2\mathds{Z}$.

Almost all the examples studied in the literature have a specific property, which defines contracting groups. On the other hand, the following are only known examples of non-contracting weakly branch groups in the literature. The groups by Dahmani \cite{Dahmani2005}, Mamaghani \cite{mamaghani2011fractal} and Noce \cite{Noce2021} are non-contracting weakly branch groups. The Hanoi towers group $\Gamma_3$ is a contracting regular branch group; whereas, its generalization $\Gamma_d$ for $d \geq 4$ is a family of non-contracting weakly branch groups \cite{Grigorchuk2006hanoi, Thesis_Skipper}. Not much about these groups are known including whether these are even branch groups. The limitation of studying the properties of these groups is attributed to their non-contracting nature. Recently, in \cite{saha2023branch}, the authors introduced a class of non-contracting self-similar groups $G_d$, for integers $d \ge 3$. For odd values of $d$, the group $G_d$ is fractal and weakly branch over its commutator subgroup. Further, it was established that the group $G_3$ is a branch group; this is the first example of explicitly constructed branch group in the class of non-contracting self-similar groups.

In this paper, we further study the class of groups $G_d$ ($d \ge 3$) introduced in \cite{saha2023branch} and focus only on odd values of $d$. The results proved in this work for $G_d$ are for all odd values of $d \ge 3$. In Section \ref{growth_Gd}, we show that the subsemigroup generated by the defining generators of $G_d$ is a trace monoid and give the presentation of the semigroup (cf. Theorem \ref{trace}). Accordingly, we observe that the group $G_d$ is of exponential growth. Further, in Section \ref{str_prop}, we explicitly calculate the structures of $k$-th level stabilizer $\text{St}_{G_d}(\widehat{k})$  (cf. Theorem \ref{stab}), $k$-th level rigid stabilizer $\text{Rist}_{G_d}(\widehat{k})$  (cf. Theorem \ref{rist}), and their quotient $\text{St}_{G_d}(\widehat{k})/\text{Rist}_{G_d}(\widehat{k})$  (cf. Theorem \ref{quotient}). Consequently, we establish that  the group $G_d$ is very strongly fractal, branch but not regular branch group. As the tools available in the literature for studying congruence subgroup problem are for regular branch groups, they are not applicable for $G_d$. Nevertheless, in Section \ref{rigid_haus}, we explicitly calculate the structure of the rigid kernel of $G_d$ (cf. Theorem \ref{rig_ker}) and observe that the branch kernel of $G_d$ is nontrivial. Moreover, we establish that the closure of $G_d$ has Hausdorff dimension arbitrarily close to 1 (cf. Theorem \ref{haus}).

Thus, the groups $G_d$ ($d \ge 3$) provided in this work appear as the first example of an infinite family of branch groups constructed explicitly in the class of non-contracting self-similar groups. Further, this family also serves as the first example of branch groups with infinite rigid kernel containing elements of infinite order.

\section{Preliminaries}

In this section, we present the concepts used in this work and fix the notation. For more details, one may refer to \cite{Bartholdi2003,Grigorchuk2005,Nekrashevych2005}.  

For a set $A$, we write $A^*$ to denote the free monoid of the words over $A$ with respect to concatenation. The length of a word $w$ is denoted by $|w|$. The set of all words of length $k$ over a set $A$ is denoted by $A^k$. For an integer $d \ge 2$, let $X = \{1, 2, \ldots, d\}$. The $d$-regular rooted tree over $X$, denoted by $T_X$ (or simply $T$), is a tree with the vertex set $X^*$, rooted at the empty word, and every vertex $u$ has exactly $d$ children $ux$, for each $x \in X$. Thus, $X^k$ is the set of vertices of $T_X$ at $k$-th level. The set of all graph automorphisms of $T$ preserving the root, written $\text{Aut}(T)$, is a group with respect to composition of maps, $gh(u) = h(g(u))$ for $g, h \in \text{Aut}(T)$ and $u \in X^*$. We shall denote the identity element of $\text{Aut}(T)$ by $e$. For $g \in \text{Aut}(T)$, note that $g$ is length preserving.

Let $u \in X^*$ and $g \in \text{Aut}(T)$. The section of $g$ at $u$, denoted by $g|_u$, is an automorphism $g|_u :T \to T$ given by  $g|_u(v) = v'$, where $v' \in X^*$ is the unique string such that $g(uv) = g(u)v'$. For any $g, h \in \text{Aut}(T)$ and $u, v \in X^*$, it is evident that $g|_{uv} = g|_u|_v$, $gh|_u = g|_uh|_{g(u)}$ and $g^{-1}|_u = (g|_{g^{-1}(u)})^{-1}$.  

Let $S_d$ be the symmetric group on $X$. The group $\text{Aut}(T)$ can be identified as the wreath product $\text{Aut}(T)\wr S_d$ through the isomorphism $\psi$ given by, for $g  \in \text{Aut}(T)$, $$\psi(g) = (g|_1, g|_2, \ldots, g|_d)\lambda_g,$$  where $\lambda_g$ is the induced action of $g$ on the set $X$. For $g \in \text{Aut}(T)$, we often write $g$ in place of $\psi(g)$ so that $g = (g|_1, g|_2, \ldots, g|_d)\lambda_g$, called the wreath recursion of $g$. 

Let $G$ be a subgroup of $\text{Aut}(T)$, written $G \le \text{Aut}(T)$. Note that $G$ acts on the tree $T$. For each $k$, the \textit{$k$-th level stabilizer}, denoted by $\text{St}_G(\widehat{k})$, is the stabilizer of the set $X^k$, i.e., $\bigcap_{u \in X^k} \text{St}_G(u)$, the intersection of the stabilizers of vertices in the $k$-th level. The normal subgroup $\text{St}_G(\widehat{k})$ is precisely the kernel of the induced action of $G$ on the finite set $X^k$, and hence it has finite index in $G$.

A subgroup $G \leq \text{Aut}(T)$ is said to be \textit{self-similar} if $g|_u \in G$, for every $g \in G$ and $u \in X^*$. 
For a self-similar group $G$, the restriction of the map $\psi$ to $G$ embeds $G$ into the wreath product $G \wr S_d$. Further, the restriction of $\psi$ to $\text{St}_G(\widehat{1})$ induces an embedding $\psi_1: \text{St}_G(\widehat{1}) \rightarrow\;  \stackrel{d}{G\times \cdots \times G}$ given by $\psi_1(g) = (g|_1, g|_2, \ldots, g|_d)$, for $g \in \text{St}_G(\widehat{1}).$ Due to self-similarity of $G$, the homomorphism $\psi_1$ can be extended to each positive integer $k \in \mathds{N}$ in a natural way such that $\psi_k: \text{St}_G(\widehat{k}) \to\; \stackrel{d^k}{G\times \cdots \times G}$ is also injective, for all $k \in \mathds{N}$. Similar to $\psi$, for $g \in \text{St}_G(\widehat{k})$, we often write $g$ in place of $\psi_k(g)$.

Let $G \leq \text{Aut}(T)$ be a self-similar group acting transitively on each level $k$, i.e., for all $u, v \in X^k$, there exists an element $g \in G$ such that $g(u) = v$. The group $G$ is said to be \textit{fractal} if  $\pi_u(\text{St}_G(u)) = G$, for all vertices $u \in X^*$, where $\pi_u: \text{St}_G(u) \to G$ is a homomorphism given by $\pi_u(g) = g|_u$, for $g \in \text{St}_G(u)$. Further, for all $k \in \mathds{N}\cup \{0\}$ and $x \in X$, if  $\pi_x(\text{St}_G(\widehat{k+1})) = \text{St}_G(\widehat{k})$, then we say that  $G$ is \textit{very strongly fractal}.

A self-similar group $G \leq \text{Aut}(T)$ is said to be \textit{contracting} if there exists a finite subset $S \subseteq G$ satisfying the following: for every $g \in G$ there is a $k \in \mathds{N}$ such that $g|_u$ belongs to $S$, for all vertices $u \in X^*$ with $|u| \geq k$; otherwise, $G$ is said to be \textit{non-contracting}.

Let $G \le \text{Aut}(T)$. The \textit{rigid stabilizer} of a vertex $u$, denoted by $\text{Rist}_G(u)$, is the subgroup of $G$ consisting the elements that fix the vertices $uv$, for all $v \in X^*$. The \textit{$k$-th level rigid stabilizer}, denoted by $\text{Rist}_G(\widehat{k})$, is the subgroup generated by the elements of $\text{Rist}_G(u)$, for all $u \in X^k$. 

Let $G \le \text{Aut}(T)$ be a group acting transitively on each level $k$. The group $G$ is said to be a \textit{branch group} if the index $[G: \text{Rist}_G(\widehat{k})]$ of $\text{Rist}_G(\widehat{k})$ in $G$ is finite, for every level $k$. Whereas, the group $G$ is said to be \textit{weakly regular branch} over a nontrivial subgroup $K \leq G$ if $\stackrel{d}{K\times \cdots \times K} \le \psi(K \cap \text{St}_G(\widehat{1}))$. Additionally, if $[G: K] < \infty$, then $G$ is said to be \textit{regular branch} over $K$. 
Note that if a group is regular branch, then it is a branch group.

We use the following notations in the rest of the paper. For elements $g, h$ of a group, we write $[g, h]$ to denote $g^{-1}h^{-1}gh$, the commutator of $g$ and $h$. For a set $A$, let $\tilde{A} = A \cup A^{-1}$, where $A^{-1} = \{a^{-1}\ :\ a \in A\}$, the set of formal inverses of elements of $A$. Let $w$ be a word over $\tilde{A}$. For $p \in A$, we write $|w|_p$ to denote the exponent sum of $p$ in $w$. Further, $|w|_{A} = \sum_{a \in A} |w|_a$. For any $j \in \mathds{N}$, we write $\overline{j} \in \{1, \ldots, d\}$ such that $\overline{j} \equiv j\  (\text{mod}\; d)$. If a word $u$ is a subword of $w$, then we write $u \ll w$; otherwise, we write $u \nll w$.

We now recall the definition and certain properties of the group $G_d$ from \cite{saha2023branch}. 
For $d \geq 3$, the group $G_d \leq \text{Aut}(T)$ is generated by $A = \{a_1, a_2, \dots, a_d\}$, where $a_i$'s are given by their wreath recursions as per the following:
\begin{eqnarray*}
	a_{1}& = &(a_{1}, a_{2}, e, \dots, e) (1\ 2)\\
	a_{2}& = &(e, a_{2}, a_{3}, e, \dots , e)(2\ 3)\\
	&\vdots&\\
	a_{d-1}& = &(e, \dots , e, a_{d-1}, a_d)(d-1\ d)\\
	a_{d}& = &(a_{1}, e, \dots, e, a_{d})(d\ 1)
\end{eqnarray*}

\begin{theorem}[\cite{saha2023branch}]\label{exponent_sum_d}		
	Let $d$ be odd. Suppose $w$ is a word over $\tilde{A}$ representing the identity element of $G_d$. Then $|w|_p = 0$, for all $p \in A$.
\end{theorem}

\begin{theorem}[\cite{saha2023branch}]\label{weakly_branch}
	For odd values of $d$, the group $G_{d}$ is fractal and weakly regular branch over its commutator subgroup $G_d'$.
\end{theorem}

In this work, $d$ always refers to any odd integer of value at least 3. For $1 \le i \le d$, consider $\xi_i = [a_{i}, a_{\overline{i+1}}][a_{\overline{i+1}}, a_{\overline{i+2}}][a_{\overline{i+1}}^{2}, a_{\overline{i+2}}]^{-1} \in G_{d}$. Recall from \cite{saha2023branch} that  
\begin{equation}\label{h_i}
	\psi(\xi_{i}) = (e, \dots, e)(i\ \overline{i+2})(\overline{i+1}\ \overline{i+3}).
\end{equation}

\section{Growth of $G_d$}
\label{growth_Gd}

In this section, we show that the semigroup generated by the generating set $A$ is a trace monoid (free partially commutative monoid) by giving a presentation.  Accordingly, we observe that the group $G_d$ is of exponential word growth.

First we give the structure of the sections of words in $A^*$ through the following lemmas.  

\begin{lemma}\label{s1}
	For $w \in A^*$, for all $i \in \{1, \ldots, d\}$, $w|_i$ is either the empty word or of the form $a_iw'$, for some $w' \in A^*$.
\end{lemma}

\begin{proof}
	Let $w \in A^*$ and $i \in \{1, \ldots, d\}$. For $j \in \{1, \ldots, d\} \setminus \{\overline{i-1}, i\}$, note that $a_j|_i = e$ and $a_j(i) = i$. Therefore, if $a_{\overline{i-1}}, a_i \nll w$, then $w|_i = e$. 
	
	Let $a_{\overline{i-1}} \ll w$ or $a_{i} \ll w$.
	Then, for $p \doteq a_{\overline{i-1}}$ or $a_i$, suppose $w$ is of the form $upv$, where $u, v \in A^*$ and $a_{\overline{i-1}}, a_i \nll u$. Note that $u|_i = e$ and $u(i) = i$. If $p = a_{\overline{i-1}}$, then $w|_i = u|_ia_{\overline{i-1}}|_{u(i)}v|_{ua_{\overline{i-1}}(i)} = a_iv|_{\overline{i-1}}$. Similarly, if $p = a_i$, then  $w|_i = a_iv|_{\overline{i+1}}$.
\end{proof}

\begin{remark}\label{empty_sec}
	For $w \in A^*$ and $1 \le i \le d$, if $w|_i$ is the empty word, then $a_{\overline{i-1}}, a_i \nll w$ and hence $\lambda_w(i) = i$.
\end{remark}

\begin{lemma}\label{s2}
	For $w \in A^*$ and $1 \le i \le d$, if $a_ja_l \ll w|_i$, then $l = \overline{j -1}$ or $\overline{j+1}$.
\end{lemma}

\begin{proof}
	Let $w|_i = w_0a_ja_lw_1$ for some $w_0, w_1 \in A^*$. Note that $w = uv$ such that $u|_i = w_0a_j$ and $v|_{u(i)} = a_lw_1$.  By Lemma \ref{s1}, it is sufficient to prove that $u(i) = \overline{j-1}$ or $\overline{j+1}$. 
	By wreath recursions of the elements of $A$, for $a_s \in A$ and $1 \le r \le d$, if $a_s|_r = a_j$, then $s \in \{\overline{j-1}, j\}$ and $r = j$. Hence, since $u|_i = w_0a_j$, $u$ is of the form $u_0pu_1$ such that $u_0|_i = w_0$, $u_0(i) = j$, $p = a_{\overline{j-1}}$ or $a_j$ and $u_{1}|_{p(j)}$ is the empty word. 
	In case $p = a_{\overline{j-1}}$, since $u_{1}|_{p(j)} = u_1|_{\overline{j-1}}$ is the empty word, by Remark \ref{empty_sec}, we have $u_1(\overline{j-1}) = \overline{j-1}$ and hence, $u(i) = \overline{j-1}$. Similarly, if $p = a_{j}$, since $u_{1}|_{p(j)} = u_1|_{\overline{j+1}}$ is the empty word, we have $u(i) = \overline{j+1}$.
\end{proof}

\begin{remark}\label{a_t}
	For $t \in \{1, \ldots, d\}$ and $u \in A^*$, if none of $a_{\overline{t-1}}, a_t$, and $a_{\overline{t+1}}$ appear in $u$, then it is evident from the wreath recursions of $a_i$'s that $a_t \nll u|_x$, $a_{\overline{t+1}} \nll u|_x$, for all $x \in \{1, \ldots, d\}$ and $\lambda_u$ fixes $t$ and $\overline{t + 1}$.
\end{remark}

\begin{theorem}\label{trace}
	The semigroup $S$ generated by $A$ is a trace monoid. In particular, $S$ has the following presentation 
	$$\langle a_1, \ldots, a_d : a_ia_j = a_j a_i, 1 < |j - i| < d - 1 \rangle .$$
\end{theorem}

\begin{proof}
	Clearly $a_ia_j = a_j a_i, 1 < |j - i| < d - 1$; we call these commuting relations as `the commutators' in this proof. For $u, v \in A^*$, if $u$ cannot be transformed to $v$ by applying the commutators, then we show that $u \neq v$ in $S$. On contrary, we assume such relations exist in $S$ and $u = v$ is such a relation with $|u| + |v|$ is minimum. Since at least one of $u$ and $v$ is nontrivial, without loss of generality say $|u| \ge 1$. Suppose $u$ is of the form $u_1a_t$, for $u_1 \in A^*$. Note that 
	\begin{eqnarray*}
		u|_{\lambda_{u_1}^{-1}(t)} &=& u_1|_{\lambda_{u_1}^{-1}(t)}a_{t}|_{\lambda_{u_1}(\lambda_{u_1}^{-1}(t))} = u_1|_{\lambda_{u_1}^{-1}(t)}a_{t}  \text{ and}\\
		u|_{\lambda_{u_1}^{-1}(\overline{t+1})} &=& u_1|_{\lambda_{u_1}^{-1}(\overline{t+1})}a_{t}|_{\lambda_{u_1}(\lambda_{u_1}^{-1}(\overline{t+1}))} = u_1|_{\lambda_{u_1}^{-1}(\overline{t+1})}a_{\overline{t+1}}.
	\end{eqnarray*}
    For all $p \in A$, since $|u|_p = |v|_p$ by Theorem \ref{exponent_sum_d}, note that $a_t \ll v$. Let $v  = v_0a_tv_1$, for $v_0, v_1 \in A^*$ such that $a_t \nll v_1$. If both $a_{t-1} \nll v_1$ and $a_{\overline{t+1}} \nll v_1$ then using the commutators $v$ can be transformed to $v_0v_1a_t$, which implies $v_0v_1 = u_1$ with $|v_0v_1| + |u_1| < |u| + |v|$, which is a contradiction. Thus, $a_{\overline{t-1}} \ll v_1$ or $a_{\overline{t+1}} \ll v_1$. 
    
    Case 1: $a_{\overline{t-1}}$ is last among $\{a_{\overline{t-1}}, a_t, a_{\overline{t+1}}\}$ in $v_1$ and therefore, $v$ is of the form $v_2a_{\overline{t-1}}v_3$, where $a_{\overline{t-1}}, a_t, a_{\overline{t+1}} \nll v_3$.
    Since $\lambda_{v} = \lambda_{u}$, we have $\lambda_{v_2a_{\overline{t-1}}v_3} = \lambda_{u_1a_t}$, and therefore, $\lambda_{v_2} = \lambda_{u_1}\ (t\; \overline{t+1}) \;\lambda_{v_3}^{-1}\;(\overline{t-1} \; t)$.
    By Remark \ref{a_t}, as $\lambda_{v_3}$ fixes $t$ and $\overline{t+1}$, $\lambda_{v_3}^{-1}$ commutes with $(t \; \overline{t+1})$. Thus,
    $\lambda_{v_2} = \lambda_{u_1} \lambda_{v_3}^{-1}\; (\overline{t-1} \; t \; \overline{t+1})$. Since $\lambda_{v_2}^{-1} = (\overline{t+1} \; t \; \overline{t-1})\lambda_{v_3}\lambda_{u_1}^{-1}$ and $\lambda_{v_3}$ fixes the vertex $\overline{t+1}$, observe that $\lambda_{v_2}^{-1}(\overline{t-1}) = \lambda_{u_1}^{-1}(\overline{t+1})$.
    Therefore, 
    \begin{eqnarray*}
    	v|_{\lambda_{u_1}^{-1}(\overline{t+1})} &=& v|_{\lambda_{v_2}^{-1}(\overline{t-1})}\\
    	&=& v_2|_{\lambda_{v_2}^{-1}(\overline{t-1})}a_{\overline{t-1}}|_{\lambda_{v_2}(\lambda_{v_2}^{-1}(\overline{t-1}))}v_3|_{\lambda_{a_{\overline{t-1}}}(\overline{t-1})}\\
    	& = & v_2|_{\lambda_{v_2}^{-1}(\overline{t-1})}a_{\overline{t-1}}v_3|_{t}.
    \end{eqnarray*}
    Since $a_t, a_{\overline{t+1}} \nll v_3|_{t}$ (by Remark \ref{a_t}), $v|_{\lambda_{u_1}^{-1}(\overline{t+1})}$ cannot be transformed to $v_4a_{\overline{t+1}}$, for $v_4 \in A^*$, by Lemma \ref{s2}. Hence, $u|_{\lambda_{u_1}^{-1}(\overline{t+1})}$ and $v|_{\lambda_{u_1}^{-1}(\overline{t+1})}$ are not identical and one cannot be transformed to other by using the commutators. Since $u|_{\lambda_{u_1}^{-1}(\overline{t+1})}$ and $v|_{\lambda_{u_1}^{-1}(\overline{t+1})}$ are nonempty words, by Lemma \ref{s1}, the word $u|_{\lambda_{u_1}^{-1}(\overline{t+1})}$ is of the form $a_{\lambda_{u_1}^{-1}(\overline{t+1})}u'$ and $v|_{\lambda_{u_1}^{-1}(\overline{t+1})}$ is of the form $a_{\lambda_{u_1}^{-1}(\overline{t+1})}v'$, for some $u', v' \in A^*$. Since $u|_{\lambda_{u_1}^{-1}(\overline{t+1})}$ and $v|_{\lambda_{u_1}^{-1}(\overline{t+1})}$ are not identical, at least one of $u'$ and $v'$ is nonempty. Since $u|_{\lambda_{u_1}^{-1}(\overline{t+1})} = v|_{\lambda_{u_1}^{-1}(\overline{t+1})}$, we have $u' = v'$ with $|u'| + |v'| < |u|_{\lambda_{u_1}^{-1}(\overline{t+1})}| + |v|_{\lambda_{u_1}^{-1}(\overline{t+1})}| \le |u|+|v|$, which contradicts the fact that $u = v$ is the relation such that $|u| + |v|$ is minimum.
    
    Case 2: Suppose $v = v_5a_{\overline{t+1}}v_6$, where $a_{\overline{t-1}}, a_t, a_{\overline{t+1}} \nll v_6$.
    Similar to Case 1, we can observe that 
    $v|_{\lambda_{u_1}^{-1}(t)} = v|_{\lambda_{v_5}^{-1}(\overline{t+2})} = v_5|_{\lambda_{v_5}^{-1}(\overline{t+2})}a_{\overline{t+2}}v_6|_{\overline{t+1}}$ and therefore, $u|_{\lambda_{u_1}^{-1}(t)}$ and $v|_{\lambda_{u_1}^{-1}(t)}$ are not identical and one cannot be transformed to other by using the commutators. By Lemma \ref{s1}, for some $u'', v'' \in A^*$, we have $u|_{\lambda_{u_1}^{-1}(t)} = a_{\lambda_{u_1}^{-1}(t)}u''$ and $v|_{\lambda_{u_1}^{-1}(t)} = a_{\lambda_{u_1}^{-1}(t)}v''$. Therefore, $u''$ cannot be transformed to $v''$ by using the commutators and $u'' = v''$ with $|u''| + |v''| <  |u|+|v|$, which is a contradiction.
\end{proof}

By Theorem \ref{trace}, it is clear that the subsemigroup generated by $a_1$ and $a_2$ is free and hence we have the following result.

\begin{corollary}
	The group $G_d$ is of exponential word growth.
\end{corollary}

\section{Structural Properties}
\label{str_prop}

In this section, we give the structure of rigid level stabilizers of $G_d$ and accordingly prove that $G_d$ is not a regular branch group. We also characterize the elements of the level stabilizers of $G_d$ and observe that $G_d$ is a very strongly fractal group. Further, we establish the structure of the quotient ${\textup{St}_{G_d}(\widehat{k})}/{\textup{Rist}_{G_d}(\widehat{k})}$ and show that $G_d$ is a branch group. Unless it is required to distinguish, in the rest of the paper, (level) stabilizers and (level) rigid stabilizers in $G_d$ will be written without $G_d$ in the subscript. 

The following remarks on the elements of $G_d$ are useful in the sequel.

\begin{remark}\label{length_twice}
	For a word $w$ over $\tilde{A}$, let the wreath recursion of $g$ be $(g_1, g_2, \ldots, g_d)\lambda_g$. Then $|g_1g_2 \cdots g_d|_A = 2|g|_A$.
\end{remark}

\begin{remark}\label{odd_even}
	Note that the symmetric group $S_d$ is generated by $\{\lambda_{a_i} : 1 \le i \le d\}$. Therefore, for any odd (or even) permutation $\lambda$ in $S_d$, there exists $g \in G_d$ such that $\lambda_g = \lambda$ and $|g|_A$ is odd (or even, respectively).
\end{remark}

\begin{remark}\label{len_even_st}
	If $|g|$ is odd, then $\lambda_g$ is not identity so that $g \not\in \text{St}(\widehat{1})$. Thus, if $g \in \text{St}(\widehat{1})$, then $|g|$ is even and hence $|g|_A$ is even. 
\end{remark}

For $k \ge 1$, consider the subset $H_k = \{w \in G_d \; :\; |w|_{A} \equiv 0\ (\text{mod}\ 2^{k+1})\}$ of $G_d$. Note that $H_k$ is a normal subgroup of $G_d$. We show that $\psi_{k}({\textup{Rist}(\widehat{k})})$ is a product of $d^k$ copies of $H_k$. First, we prove the following lemma.

\begin{lemma}\label{lambda}
	For $g_1, \ldots, g_d \in G_d$, if $(g_1, \ldots, g_d) \in \psi(\textup{St}(\widehat{1}))$, then for any $\lambda \in S_d$, $(g_{\lambda(1)}, \ldots, g_{\lambda(d)}) \in \psi(\textup{St}(\widehat{1}))$.
\end{lemma}

\begin{proof}
	Let $\psi{(g)} = (g_1, \ldots, g_d) \in \psi(\textup{St}(\widehat{1}))$. Since $\{\lambda_{a_1}, \ldots, \lambda_{a_d}\}$ generates $S_d$, for any $\lambda \in S_d$, there exists $h \in G_d$ such that $\psi(h) = (h|_1, \ldots, h|_d)\lambda$. Observe that 
	\begin{eqnarray*}
		hgh^{-1} & = & (h|_1g|_{\lambda(1)}h^{-1}|_{\lambda(1)}, \ldots, h|_dg|_{\lambda(d)}h^{-1}|_{\lambda(d)})\\
		& = & (h|_1g_{\lambda(1)}{h|_{\lambda^{-1}(\lambda(1))}}^{-1}, \ldots, h|_dg_{\lambda(d)}{h|_{\lambda^{-1}(\lambda(d))}}^{-1})\\
		& = & (h|_1g_{\lambda(1)}h|_1^{-1}, \ldots, h|_dg_{\lambda(d)}h|_d^{-1})
	\end{eqnarray*}
	By Theorem \ref{weakly_branch}, since $G_d' \times \ldots \times G_d' \le \psi(G_d' \cap \textup{St}(\widehat{1}))$, we have
	$$(h|_1{g_{\lambda(1)}}^{-1}{h|_1}^{-1}g_{\lambda(1)}, \ldots, h|_d{g_{\lambda(d)}}^{-1}{h|_d}^{-1}g_{\lambda(d)}) \in \psi(\textup{St}(\widehat{1})).$$ Hence, 
	$(g_{\lambda(1)}, \ldots, g_{\lambda(d)}) \in \psi(\textup{St}(\widehat{1}))$.
\end{proof}

The structure of $k$-th level rigid stabilizer of $G_3$ was established in \cite[Theorem 5.8]{saha2023branch}. We now extend the result for $G_d$ in the following theorem.  

\begin{theorem}\label{rist}
	For all $k \in \mathds{N}$, $\psi_k(\textup{Rist}(\widehat{k})) = \stackrel{d^k}{H_k \times \dots \times H_k}$.
\end{theorem}

\begin{proof}
	In view of \cite[Theorem 5.8]{saha2023branch}, let $d \ge 5$. We prove the result by induction on $k$. 
	For $k = 1$, we show that $\psi(\text{Rist}(1)) = H_1 \times \{e\} \times \ldots \times \{e\}$. Hence, by Lemma \ref{lambda}, we get $\psi(\text{Rist}(\widehat{1})) = H_1 \times \cdots \times H_1$.
	
	Let $\psi(\text{Rist}(1)) = K \times \{e\} \times \ldots \times \{e\}$. We show that $K = H_1$. First note that $K$ is a subgroup of $G_d$. Observe that $a_1^{-1}a_4^{-1}a_2^{-1}a_1a_3a_4\xi_2\xi_1 = (a_2^{-1}a_4, e, \ldots, e)$ and $(a_1^{a_2}\xi_1)^{2} = ((a_1a_2)^{2}, e, \ldots, e)$. Thus, $a_2^{-1}a_4, (a_1a_2)^{2}\in K$. Since $G_d$ is weakly regular branch group over $G_{d}'$, we have $G_d' \le K$ so that $a_1^2a_2^2 \in K$. By symmetry, note that, for $1 \le i \le d$,
	$a_i^{-1}a_{\overline{i+2}}, a_i^2a_{\overline{i+1}}^2 \in K$. Accordingly, for $1 \le i, j \le d$, $a_i^{-1}a_j, a_i^2a_j^2 \in K$, and hence, $a_i^{4}$ is also in $K$.
	
	Let $w\in H_1$ and, for $1 \le i \le d$, $|w|_{a_i} = n_{i}$. Thus, $n_{1} + \ldots + n_{d} \equiv 0\; (\text{mod}\ 4)$. Observe that 
	\begin{eqnarray*}
		w &=& a_1^{n_{1}} \cdots a_d^{n_{d}}h_1,\ \text{for some}\ h_1 \in G_d'\\
		&=& a_1^{n_{1}} \cdots a_{d-1}^{n_{d-1}}a_d^{-n_{1} - \ldots - n_{d-1} + 4t}h_1,\ \text{for some}\;  t \in \mathds{Z}\\
		&=& (a_d^{-n_{1}}a_1^{n_{1}})(a_d^{-n_2}a_2^{n_{2}}) \cdots (a_d^{-n_{d-1}}a_{d-1}^{n_{d-1}})h_2,\\ 
		&& \qquad \qquad \text{for some}\ h_2 \in K, \text{as}\; G_d' \le K\; \text{and}\; a_d^4 \in K\\
		&=& (a_d^{-1}a_1)^{n_{1}} \cdots (a_d^{-1}a_{d-1})^{n_{d-1}}h_3,\ \text{for some}\ h_3 \in K.
	\end{eqnarray*}
	Hence, since $a_d^{-1}a_i$'s are in $K$, we have $w \in K$ so that $H_1 \subseteq K$.
	
	To show $K \subseteq H_1$, let $w \in K$ with $|w|_{a_i} = m_{i}$, for $1 \le i \le d$. If possible, suppose $w \notin H_1$. Then, $m_1 + \ldots + m_d = s + 4t$, for some $t \in \mathds{Z}$ and $0 < s \le 3$. Note that $w = a_1^{m_{1}} \cdots a_d^{m_{d}}g_1$ for some $g_1 \in G_d'$. Hence, $w = a_1^sg_2$, where  $g_2 = a_1^{4t}(a_d^{-m_d}a_{d-1}^{-m_{d-1}} \cdots a_2^{-m_2}a_1^{m_2 + \ldots + m_d})^{-1}g_1 \in K$, so that $a_1^s \in K$.  
 However, we show that $a_1^s \notin K$ for all $s = 1, 2, 3$, leading to a contradiction. If $a_1, a_1^3 \in K$, then there exist $g_1, g_2 \in G_d$ such that $g_1 = (a_1, e, \ldots, e)$, $g_2 = (a_1^3, e, \ldots, e)$. Then, by Remark \ref{length_twice}, $2|g_1|_A = 1$ and $2|g_2|_A = 3$; which are not possible. Further, if $a_1^2 \in K$, then for $h = (a_1^2, e, \ldots, e) \in G_d$, we have $|h|_A = 1$ by Remark \ref{length_twice}. But this is not possible by Remark \ref{len_even_st}, as $h \in \text{St}(\widehat{1})$.
	Hence, $K \subseteq H_1$ so that $\text{Rist}(1) = H_1 \times \{e\} \times \cdots \times \{e\}$.
	
	Suppose $\textup{Rist}(\widehat{k}) = \stackrel{d^k}{H_k \times \dots \times H_k}$, for some $k \in \mathds{N}$. Thus, $\textup{Rist}(\underset{\text{length } k}{11\ldots 1}) = H_k \times \{e\} \stackrel{d^k-1}{\times \dots \times }\{e\}$.
	We show that $\textup{Rist}(\underset{\text{length }k+1}{11\ldots 1}) = H_{k+1} \times \{e\} \stackrel{d^{k+1}-1}{\times \dots \times} \{e\}$. 	
	Let $g\in \textup{Rist}(\underset{\text{length }k+1}{11\ldots 1})$. Clearly, $g \in \textup{Rist}(\underset{\text{length }k}{11\ldots 1})$. Then there exist $g', g'' \in G_d$ such that $\psi_{k+1}(g) = (g', e, \ldots, e)$ and $\psi_{k}(g) = (g'', e, \ldots, e)$ with $\psi_1(g'') = (g', e, \ldots, e)$. Hence, by inductive hypothesis, we have $g'' \in H_k$ so that $|g''|_{A} \equiv 0\ (\textup{mod}\ 2^{k+1})$. Since $\psi(g'') = (g', e, \ldots, e)$, by Remark \ref{length_twice}, we have $|g'|_{A} = 2|g''|_{A}$ so that $|g'|_{A} \equiv 0\ (\textup{mod}\ 2^{k+2})$. Hence, $g' \in H_{k+1}$ so that $\textup{Rist}(\underset{\text{length } k+1}{11\ldots 1}) \subseteq H_{k+1} \times \{e\} \stackrel{d^{k+1}-1}{\times \dots \times} \{e\}$. 
	
	Conversely, suppose $h \in H_{k+1}$. Clearly, $h \in H_1$ so that there exists $h' \in \textup{Rist}(1)$ such that $\psi(h') = (h, e, \ldots, e)$. Since $|h|_{A} \equiv 0\ (\textup{mod}\ 2^{k+2})$, again by Remark \ref{length_twice}, $|h'|_{A} \equiv 0\ (\textup{mod}\ 2^{k+1})$, i.e., $h' \in H_k$. Therefore, by inductive hypothesis, there exists $g \in \textup{Rist}(\underset{\text{length } k}{11\ldots 1})$ such that $\psi_k(g) = (h', e, \ldots, e)$ and thus $\psi_{k+1}(g) = (h, e, \ldots, e)$. Hence, $\textup{Rist}(\underset{\text{length } k+1}{11\ldots 1}) = H_{k+1} \times \stackrel{d^{k+1}-1}{ \{e\} \times \dots \times \{e\}}$. Consequently, by Lemma \ref{lambda}, we have $\textup{Rist}(\widehat{k+1}) = \stackrel{d^{k+1}}{H_{k+1} \times \dots \times H_{k+1}}$.
\end{proof}

To establish that $G_d$ is not a regular branch group, we recall the following property of regular branch groups.

\begin{proposition}[\cite{francoeur2024structure}]\label{prop_12}
	Let $G \le \textup{Aut}(T)$ be a regular branch group over a finite index subgroup $K \le G$. Then there exist $k_0 \in \mathds{N}$ and a finite index normal subgroup $L$ of $G$ such that $K \le L$ and $\pi_v(\textup{Rist}_G(v)) = L$, for all $v \in X^k$ and $k \ge k_0$.
\end{proposition}

For $k \in \mathds{N}$, by Theorem \ref{rist}, note that $\pi_v(\textup{Rist}(v)) = H_k$ for all $v \in X^k$. Since $H_k$'s are distinct for different $k$'s, in view of Proposition \ref{prop_12}, we have the following corollary of Theorem \ref{rist}.

\begin{corollary}
	The group $G_d$ is not a regular branch group.
\end{corollary}

We now proceed to characterize the elements of $k$-th level stabilizers of $G_d$. For that, we need the following lemma.

\begin{lemma}\label{lem_h_i}
	For any $h \in G_d$ and for $1 \le j < l \le d$, the element
	$$(e, \ldots, e, \underset{j}{h}, e, \ldots, e, \underset{l}{\overline{h}}^{^{-1}}, e, \ldots, e) \in \textup{St}(\widehat{1}),$$ where $\overline{h}$ is the reverse word of $h$.
\end{lemma}

\begin{proof}
	From Equation (\ref{h_i}), recall that, for $i \in \{1, \ldots, d\}$, 
	   $$\xi_{i} = (e, \ldots, e)(i\ \overline{i+2})(\overline{i+1}\ \overline{i+3}).$$
	Accordingly, note that 
	\begin{eqnarray*}
		\lambda_{\xi_{d-2}\xi_{d-4} \cdots \xi_{3}} & = & (3 \ 5 \ \ldots \ d-2 \ d)(1 \ 4 \ 6 \ \ldots \ d-3 \ d-1);\\
		\lambda_{\xi_{2}\xi_{4} \cdots \xi_{d-1}} & = & (1 \ d-1 \ d-3 \ \ldots \ 6 \ 4 \ d \ d-2 \ \ldots \ 5 \ 3 \ 2).
	\end{eqnarray*}
    Now, for $\eta_1 = \xi_{d-2}\xi_{d-4} \cdots \xi_{3}\xi_{2}\xi_{4} \cdots \xi_{d-1}\xi_{d}$, we have $\lambda_{\eta_1} = (1 \ 2 \ 3)$ and hence, $\eta_1 = (e, \ldots, e)(1\ 2 \ 3)$.
    By symmetry, there exist $\eta_2, \ldots, \eta_d \in G_d$ such that $\eta_i = (e, \dots, e)(i\; \overline{i+1}\; \overline{i+2})$, for $i \in \{2, \ldots, d\}$.
    For $i \in \{1, \ldots, d\}$, observe that $$[a_i, a_{\overline{i+1}}] = (e, \ldots, e, \underset{i}{a_{\overline{i+1}}^{-1}}, \underset{\overline{i+1}}{a_{\overline{i+1}}}, e, \ldots, e) (i\; \overline{i+1}\; \overline{i+2}).$$ Thus, $$[a_i, a_{\overline{i+1}}]\eta_i^{-1} = (e, \ldots, e, \underset{i}{a_{\overline{i+1}}^{-1}}, \underset{\overline{i+1}}{a_{\overline{i+1}}}, e, \ldots, e) \in \textup{St}(\widehat{1}).$$ By Lemma \ref{lambda}, for $1 \le i \le d$ and $1 \le j < l \le d$, $(e, \ldots, e, \underset{j}{a_{i}}, e, \ldots, e, \underset{l}{a_{i}^{-1}}, e, \ldots, e) \in \textup{St}(\widehat{1})$. Therefore, for any $h \in G_d$, since $h$ is a word over $a_i$'s and their inverses, we have the result.
\end{proof}

\begin{theorem}\label{stab}
	For $h_1, h_2, \ldots, h_{d^k} \in G_d$, an element $(h_1, h_2, \ldots, h_{d^k}) \in \textup{St}(\widehat{k})$ if and only if for $1\le r\le k$ and $0 \le t < d^{k - r}$, 
	$$\sum_{i = 1}^{d^r}|h_{d^rt + i}|_A \equiv 0 \ (\textup{mod}\ 2^{r+1}).$$
\end{theorem}

\begin{proof}
	We prove the statement by induction on $k$.
	
	Base case: Let $(h_1, h_2, \ldots, h_{d}) \in \textup{St}(\widehat{1})$, i.e., there exists $h \in \textup{St}(\widehat{1})$ such that $\psi(h) = (h_1, h_2, \ldots, h_{d})$. By Remark \ref{len_even_st}, $|h|_A$ is even. Since $\sum_{i = 1}^{d}|h_i|_A = 2|h|_A$ (cf. Remark \ref{length_twice}), we have $\sum_{i = 1}^{d}|h_i|_A \equiv 0 \; (\textup{mod}\; 4)$.
	
	Conversely, suppose $h_1, h_2, \ldots, h_{d} \in G_d$ such that $\sum_{i = 1}^{d}|h_i|_A \equiv 0 \; (\textup{mod}\; 4)$. By Lemma \ref{lem_h_i}, for $i \in \{1, \ldots, d-1\}$, $(e, \ldots, e, \underset{i}{h_i}, e, \ldots, e, \overline{h}_{i}^{^{-1}}) \in \textup{St}(\widehat{1})$, where $\overline{h}_{i}$ is the reverse word of $h_i$. Thus, $(h_1, h_2, \ldots, h_{d-1}, \overline{h}_{1}^{^{-1}} \cdots \overline{h}_{d-1}^{^{-1}}) \in \textup{St}(\widehat{1})$. Since $\sum_{i = 1}^{d}|h_i|_A \equiv 0 \; (\textup{mod}\; 4)$, we have $|\overline{h}_{d-1} \cdots \overline{h}_{1}h_{d}|_A \equiv 0 \; (\textup{mod}\; 4)$. Therefore, $(e, \ldots, e, \overline{h}_{d-1} \cdots \overline{h}_{1}h_{d}) \in \text{Rist}(\widehat{1})$. Hence, $(h_1, h_2, \ldots, h_{d}) \in \textup{St}(\widehat{1})$.
	
	Inductive case: Let $(h_1, h_2, \ldots, h_{d^{k+1}}) \in \textup{St}(\widehat{k + 1})$, i.e., there exists $h \in \textup{St}(\widehat{k + 1})$ such that $\psi_{k+1}(h) = (h_1, h_2, \ldots, h_{d^{k+1}})$. Observe that $\psi_k(h) = (g_1, g_2, \ldots, g_{d^k})$, for some $g_1, g_2, \ldots, g_{d^k} \in G_d$ such that 
	$$\psi(g_i) = (h_{(i-1)d + 1}, h_{(i-1)d + 2}, \ldots, h_{(i-1)d + d}).$$
	By Remark \ref{length_twice}, we have $2|g_i|_A = \sum_{j=1}^{d}|h_{(i-1)d + j}|_A$.
    Since $h \in \textup{St}(\widehat{k})$, by induction hypothesis, for $1\le r\le k$ and $0 \le t < d^{k - r}$, we have 
	$$\sum_{i = 1}^{d^r}|g_{d^rt + i}|_A \equiv 0 \; (\textup{mod}\; 2^{r+1}).$$
	\begin{tabular}{c|r}
		\quad\qquad&\begin{minipage}[r]{11.7cm}
			That is, $$\sum_{i = 1}^{d^r}2|g_{d^rt + i}|_A \equiv 0 \; (\textup{mod}\; 2^{r+2}).$$
			Since $2|g_i|_A = \sum_{j=1}^{d}|h_{(i-1)d + j}|_A$ (by Remark \ref{length_twice}), we have 
			$$\sum_{i = 1}^{d^r} \sum_{j=1}^{d}|h_{d^{r+1}t + (i-1)d + j}|_A \equiv 0 \; (\textup{mod}\; 2^{r+2})$$ and thus, $$\sum_{i = 1}^{d^{r+1}}|h_{d^{r+1}t + i}|_A \equiv 0 \; (\textup{mod}\; 2^{r+2}).$$
			Hence, for $2 \le r' = r+1 \le k+1$ and $0 \le t < d^{k + 1 - r'}$, we have
			$$\sum_{i = 1}^{d^{r'}}|h_{d^{r'}t + i}|_A \equiv 0 \; (\textup{mod}\; 2^{r'+1}).$$
		\end{minipage}
	\end{tabular}
	 In fact, these left-marked statements are equivalent and will be used in the converse. Further, for $1 \le i \le d^{k}$,  $\sum_{j=1}^{d}|h_{(i-1)d + j}|_A \equiv 0 \; (\textup{mod}\; 4)$ from base case as $g_i \in \textup{St}(\widehat{1})$. Hence, for $1\le r' \le k+1$ and $0 \le t < d^{k + 1 - r'}$,    $$\sum_{i = 1}^{d^{r'}}|h_{d^{r'}t + i}|_A \equiv 0 \ (\textup{mod}\ 2^{r'+1}).$$
	
	Conversely, let $h_1, \ldots, h_{d^{k+1}} \in G_d$ such that for $1\le r\le k+1$ and $0 \le t < d^{k + 1 - r}$,
		$$\sum_{i = 1}^{d^r}|h_{d^rt + i}|_A \equiv 0 \ (\textup{mod}\ 2^{r+1}).$$
	For $r = 1$, note that $\sum_{i = 1}^{d}|h_{dt + i}|_A \equiv 0 \ (\textup{mod}\ 4)$. Thus, for each $j \in \{1, \ldots, d^k\}$, there exists $g_j \in \textup{St}(\widehat{1})$ such that $\psi(g_j) = (h_{(j-1)d+1}, h_{(j-1)d+2}, \ldots, h_{(j-1)d+d})$. Note that, for all $j \in \{1, \ldots, d^k\}$, $2|g_j|_A = \sum_{i = 1}^{d}|h_{(j-1)d + i}|_A$. From the above marked statements, for $1\le r\le k$ and $0 \le t < d^{k - r}$, we have 
	$$\sum_{i = 1}^{d^r}|g_{d^rt + i}|_A \equiv 0 \ (\textup{mod}\ 2^{r+1}).$$  
    Hence by induction hypothesis, there exists $g \in \textup{St}(\widehat{k})$ such that $\psi_k(g) =\break (g_1, g_2, \ldots, g_{d^k})$. For each $j \in \{1, 
	\ldots, d^k\}$, since $g_j \in \textup{St}(\widehat{1})$, we have $g \in \textup{St}(\widehat{k+1})$ such that $\psi_{k+1}(g) = (h_1, h_2, \ldots, h_{d^{k+1}})$. 
\end{proof}

\begin{corollary}
	The group $G_d$ is very strongly fractal.
\end{corollary}

\begin{proof}
	Let $k \in \mathds{N}\cup \{0\}$ and $x \in X$. Suppose $g \in \text{St}(\widehat{k})$. Since $g \in \text{St}(\widehat{1})$, $|g|_A$ is even. Therefore, $|g|_A \equiv 0$ or $2 \ (\text{mod }4)$; thus, by Theorem \ref{stab}, there exists $g' \in \text{St}(\widehat{1})$ such that $\psi(g') = (e, \ldots, e, \underset{x}{g}, e, \ldots, e)$ or $\psi(g') = (e, \ldots, e, \underset{x}{g}, g, e, \ldots, e)$, respectively. In any case, since $g \in \text{St}(\widehat{k})$, we have $g' \in \text{St}(\widehat{k+1})$. Hence, $g = \pi_x(g') \in \pi_x(\text{St}(\widehat{k+1}))$.  In any self-similar group, note that  $\pi_x(\text{St}(\widehat{k+1})) = \text{St}(\widehat{k+1})|_x   \subseteq \text{St}(\widehat{k})$. Hence, $G_d$ is very strongly fractal.
\end{proof}

The following lemma is useful in the subsequent theorem on the structure of the quotient $\textup{St}(\widehat{k})/ \textup{Rist}(\widehat{k})$.

\begin{lemma}\label{ai}
	For every $g \in G_d$, there exist $h\in H_k$ and $n \in \mathds{N}$ with $0 \le n  < 2^{k+1}$ such that $g = a_1^nh$.
\end{lemma}

\begin{proof}
	For $w \in G_d'$, since $|w|_p = 0$ for all $p \in A$, we have $G_d' \le H_k$. Accordingly, for any $g \in G_d$, we have
	\begin{eqnarray*}
		g &=& a_1^{n_{1}} \cdots a_d^{n_{d}}h_1,\ \text{for some}\ h_1 \in G_d'\\
		&=& a_1^{n_{1}} \cdots a_d^{n_{d}}(a_d^{-n_{d}}\cdots a_2^{-n_{2}}a_1^{n_2 + \ldots + n_d})h_2,\\ 
		&&\text{where}\ h_2 = (a_d^{-n_{d}}\cdots a_2^{-n_{2}}a_1^{n_2 + \ldots + n_d})^{-1}h_1 \in H_k\\
		&=& a_1^{(n_{1} + \ldots + n_{d})}h_2\\
		&=& a_1^nh,\ \text{for some}\ h \in H_k,\ 0 \leq n < 2^{k+1}\ \text{with}\ n \equiv n_{1} + \ldots + n_{d} \ (\text{mod}\ 2^{k+1}).
	\end{eqnarray*}
\end{proof}

\begin{theorem}\label{quotient}
	Let $\mathcal{C}_n$ be the cyclic group $\mathds{Z}/n\mathds{Z}$. Then, for $k \ge 1$, we have ${\textup{St}(\widehat{k})}/{\textup{Rist}(\widehat{k})} \cong \prod_{i = 1}^{d^k - 1}\mathcal{C}_{\alpha_i}$ (say, $H$), where 
	$\alpha_i = 2^{k+1}$ if $d \ndivides i$; else, $\alpha_i = 2^{k-s}$, for $i = d^sr$ with $s \ge 1$ and $d \ndivides r$. Hence, $[{\textup{St}(\widehat{k})} : {\textup{Rist}(\widehat{k})}] = 2^t$, where $t = {\frac{(k+1)d^{k+1} - (k+3)d^k +d^{k-1} + 1}{d-1}}$.
\end{theorem}

\begin{proof}
	For $g \in \textup{St}(\widehat{k})$, let $\psi_{k}(g) = (g_1, \ldots, g_{d^k})$. By Lemma \ref{ai}, for every $g_i$ there exist $h_i \in H_k$ and $n_i \in \mathds{N}$ with $0 \le n_i  < 2^{k+1}$ such that $g_i = a_1^{n_i}h_i$. By Theorem \ref{rist}, $(h_1, \ldots, h_{d^k}) \in \textup{Rist}(\widehat{k})$. Since $(a_1^{n_1}h_1, \ldots, a_1^{n_{d^k}}h_{d^k})$ and $(h_1^{-1}, \ldots, h_{d^k}^{-1})$ are in $\textup{St}(\widehat{k})$, we have $(a_1^{n_1}, \ldots a_1^{n_{d^k}}) \in \textup{St}(\widehat{k})$. Therefore, by Theorem \ref{stab}, for $1\le r\le k$ and $0 \le t < d^{k - r}$, $\sum_{i = 1}^{d^r}|n_{d^rt + i}|_A \equiv 0 \ (\textup{mod}\ 2^{r+1})$. Hence, for every $g \in \textup{St}(\widehat{k})$, there exist $n_1, \ldots, n_{d^k} \in \{0, \ldots, 2^{k+1}-1\}$ and $h_1, \ldots, h_{d^k} \in \textup{Rist}(\widehat{k})$ such that $\psi_{k}(g) = (a_1^{n_1}, \ldots a_1^{n_{d^k}})(h_1, \ldots, h_{d^k})$, where, $\sum_{i = 1}^{d^r}|n_{d^rt + i}|_A \equiv 0 \ (\textup{mod}\ 2^{r+1})$, for $1\le r\le k$ and $0 \le t < d^{k - r}$. Therefore,
	\begin{equation}\label{quot_elts}
	\textup{St}(\widehat{k})/\textup{Rist}(\widehat{k})  = \Big\{(a_1^{n_1}, \ldots, a_1^{n_{d^{k}}})\textup{Rist}(\widehat{k}) : 0 \le n_j < 2^{k+1}, 
	\end{equation}
	\begin{equation*}
	\quad \quad \quad \quad \quad \quad \quad \quad \quad \quad \quad \quad \quad  \sum_{i = 1}^{d^r}n_{d^rt + i} \equiv 0 \ (\textup{mod}\ 2^{r+1}),
		1\le r\le k, 0 \le t < d^{k - r}\Big\}.
	\end{equation*}
	Thus, for the elements of $\textup{St}(\widehat{k})/\textup{Rist}(\widehat{k})$, it is sufficient to consider the powers of $a_1$. Accordingly, we denote an element $(a_1^{n_1}, \ldots, a_1^{n_{d^{k}}})\textup{Rist}(\widehat{k}) \in \textup{St}(\widehat{k})/\textup{Rist}(\widehat{k})$ by $\llbracket n_j \rrbracket_{1 \le j \le d^k}$, or simply by $\llbracket n_j \rrbracket$.
	Consider the map 
	$$\theta : {\textup{St}(\widehat{k})}/{\textup{Rist}(\widehat{k})} \to H$$
	 which sends $\llbracket n_j \rrbracket$ to $(l_1, \ldots, l_{d^k-1}) \in H$, where 
	\[l_j = \begin{cases}
	n_j & \text{ when } d \ndivides j\\
	\frac{\sum_{i=1}^{d^s}n_{(r-1)d^s+i}}{2^{s+1}}\ (\textup{mod}\ 2^{k-s}) & \text{ when } j = d^sr, s \ge 1, d \ndivides r.
	\end{cases}\] 
     For $\llbracket m_j \rrbracket, \llbracket n_j \rrbracket \in {\textup{St}(\widehat{k})}/{\textup{Rist}(\widehat{k})}$, note that $\llbracket m_j \rrbracket \llbracket n_j \rrbracket = \llbracket \gamma_j \rrbracket$, where $0 \le \gamma_j (\equiv m_j + n_j (\textup{mod}\ 2^{k+1})) < 2^{k+1}$. For $1 \le d^sr < d^k$ such that $s \ge 1$ and $d \ndivides r$, observe that $$\frac{\sum_{i=1}^{d^s}\gamma_{(r-1)d^s+i}}{2^{s+1}} \equiv \frac{\sum_{i=1}^{d^s}m_{(r-1)d^s+i}}{2^{s+1}} + \frac{\sum_{i=1}^{d^s}n_{(r-1)d^s+i}}{2^{s+1}}\ (\textup{mod}\ 2^{k-s})$$ and hence, $\theta$ is a homomorphism.
	
	Now we show that $\theta$ is onto. Let $(l_1, l_2, \ldots, l_{d^k-1}) \in H$. For each $j \in \{1, \ldots, d^k\}$, consider the numbers $0 \le n_j < 2^{k+1}$ such that $n_j = l_j$, when $d \ndivides j$ and for $1\le s \le k-1$ and $1 \le r <d^{k-s}$ with $d \ndivides r$, $n_{d^sr} \equiv 2^{s+1}l_{d^sr} - \sum_{i=1}^{d^s - 1}n_{(r-1)d^s+i} \; (\textup{mod}\ 2^{k+1})$, and $n_{d^k} \equiv - \sum_{i=1}^{d^k - 1}n_{i}\; (\textup{mod}\ 2^{k+1})$. Observe that $\llbracket n_j \rrbracket \in \textup{St}(\widehat{k})/\textup{Rist}(\widehat{k})$ and
	\[l_j = \begin{cases}
	n_j & \text{ when } d \ndivides j\\
	\frac{\sum_{i=1}^{d^s}n_{(r-1)d^s+i}}{2^{s+1}}\ (\textup{mod}\ 2^{k-s}) & \text{ when } j = d^sr, s \ge 1, d \ndivides r.
	\end{cases}\]
	Hence, $\theta$ is onto.
	
    To show that $\theta$ is one-one, let $\llbracket m_j \rrbracket, \llbracket n_j \rrbracket \in {\textup{St}(\widehat{k})}/{\textup{Rist}(\widehat{k})}$ such that $\theta(\llbracket m_j \rrbracket) = \theta(\llbracket n_j \rrbracket) = (l_1, \ldots, l_{d^{k}-1})$.
    Clearly, for $d \ndivides j$, $m_j = n_j$. Also, for $j = d^sr$, with $s \ge 1$, $d \ndivides r$, we have, 
    $\frac{\sum_{i=1}^{d^s}m_{(r-1)d^s+i}}{2^{s+1}} \equiv \frac{\sum_{i=1}^{d^s}n_{(r-1)d^s+i}}{2^{s+1}}\ (\textup{mod}\ 2^{k-s}),$
    which implies that
    \begin{eqnarray}\label{eqn}
    	\sum_{i=1}^{d^s}m_{(r-1)d^s+i} \equiv \sum_{i=1}^{d^s}n_{(r-1)d^s+i}\; (\textup{mod}\ 2^{k+1})
    \end{eqnarray}
    When $s = 1$, since, for $d \ndivides j$, $m_j = n_j$, by Equation (\ref{eqn}), we get, $m_{dr} \equiv n_{dr}(\textup{mod}\ 2^{k+1})$, and therefore, for $1 \le r < d^{k-1}$ such that $d \ndivides r$, we have $m_{dr} = n_{dr}$. Further, when $s = 2$, since, for $d \ndivides j$ or $j = dr$ such that $d \ndivides r$, $m_j = n_j$, using Equation (\ref{eqn}), we get, $m_{d^2r} = n_{d^2r}$. Successively, for each $j \in \{1, \ldots, d^{k}\}$, we have $m_j = n_j$.
    
    Further, observe that, the number of $j$'s in $\{1, \ldots, d^k-1\}$ such that $d \ndivides j$ is $(d-1)d^{k-1}$ and number of $j$'s of the form $d^sr$ such that $d \ndivides r$ and $s \ge 1$ is $(d-1)d^{k-s}$. Therefore, 
    $$H = \prod_{i = 1}^{d^k - 1}\mathcal{C}_{\alpha_i} \cong \mathcal{C}_{2^{k+1}}^{(d-1)d^{k-1}} \times \mathcal{C}_{2^{k-1}}^{(d-1)d^{k-2}} \times \mathcal{C}_{2^{k-2}}^{(d-1)d^{k-3}} \times \cdots \times \mathcal{C}_{2}^{(d-1)}.$$
    Thus,
    \begin{eqnarray*}
    	[{\textup{St}(\widehat{k})} : {\textup{Rist}(\widehat{k})}] &=& {(2^{k+1})}^{(d-1)d^{k-1}}({(2^{k-1})}^{(d-1)d^{k-2}}{(2^{k-2})}^{(d-1)d^{k-3}} \cdots 2^{d-1})\\
    	&=& {(2^{k+1})}^{(d-1)d^{k-1}}2^{(d-1)\{(k-1)d^{k-2}+(k-2)d^{k-3} + \ldots + 2d + 1\}}
    \end{eqnarray*} 
    Assuming $s = (k-1)d^{k-2}+(k-2)d^{k-3} + \ldots + 2d + 1$, we have,
    \begin{eqnarray*}
    	(d-1)s & = & (k-1)d^{k-1} - d^{k-2} - d^{k-3} - \ldots - d - 1 ,\\
    	& = & (k-1)d^{k-1} - \frac{d^{k-1}-1}{d-1}\\
    \end{eqnarray*}
    Hence,
    \begin{eqnarray*}
    	[{\textup{St}(\widehat{k})} : {\textup{Rist}(\widehat{k})}] & = & {(2^{k+1})}^{(d-1)d^{k-1}}2^{(k-1){d^{k-1}} - \frac{d^{k-1}-1}{d-1}}\\
    	&=& 2^{\frac{(k+1)d^{k+1} - (k+3)d^k +d^{k-1} + 1}{d-1}}.
    \end{eqnarray*}
\end{proof}

Note that, by Theorem \ref{quotient}, $\textup{Rist}(\widehat{k})$ is of finite index in $\textup{St}(\widehat{k})$, for all $k \in \mathds{N}$. Further, since $\textup{St}(\widehat{k})$ is of finite index in $G_d$, the index of $\textup{Rist}(\widehat{k})$ in $G_d$ is finite, for all $k \in \mathds{N}$, so that $G_d$ is a branch group as stated in the following result. 

\begin{theorem}
	The  group $G_d$ is a branch group.
\end{theorem}

\section{Rigid Kernel and Hausdorff Dimension}
\label{rigid_haus}

In this section, to study the congruence subgroup problem of the group $G_d$, we give an explicit structure of its rigid kernel. Also, we observe that the branch kernel is non-trivial.  Further, we determine the Hausdorff dimension of the closure of $G_d$. For more details on Hausdorff dimension, one may refer to \cite{Garrido_2017, Skipper2020}.

For a branch group $G$, the CSP is equivalent to  that every subgroup of finite index in $G$ contains $\textup{Rist}_G(\widehat{k})$ for some $k$, and every $\textup{Rist}_G(\widehat{k})$ contains a level stabilizer. Thus, in terms of profinite completions, $G$ has CSP if and only if its branch kernel, i.e.,  $\ker(\widehat{G} \twoheadrightarrow \widetilde{G})$, and rigid kernel, i.e., $\ker(\widetilde{G} \twoheadrightarrow \overline{G})$ are trivial; here, $\overline{G} = \underset{k \ge 1}{\underleftarrow{\text{lim}}}\ G/ \textup{St}_G(\widehat{k})$, $\widetilde{G} = \underset{k \ge 1}{\underleftarrow{\text{lim}}}\ G/ \textup{Rist}_G(\widehat{k})$, and $\widehat{G} = \underset{N \in \mathscr{N}}{\underleftarrow{\text{lim}}}\ G/ N$, where $\mathscr{N}$ is the set of finite index normal subgroups of $G$. Note that the rigid kernel of $G$, $$\text{ker}(\widetilde{G} \twoheadrightarrow \overline{G}) = \underset{k \ge 1}{\underleftarrow{\text{lim}}}\ \textup{St}_G(\widehat{k})/ \textup{Rist}_G(\widehat{k}),$$ where the maps $\rho_{k+p, k} : \textup{St}_G(\widehat{k+p})/ \textup{Rist}_G(\widehat{k+p}) \to \textup{St}_G(\widehat{k})/ \textup{Rist}_G(\widehat{k})$ of the inverse system $\Big(\textup{St}_G(\widehat{k})/ \textup{Rist}_G(\widehat{k}), \rho_{k+p, k}\Big)_{k, p \in \mathds{N}}$ come from the natural inclusions \break $\textup{St}_G(\widehat{k + p}) \hookrightarrow \textup{St}_G(\widehat{k})$ and $\textup{Rist}_G(\widehat{k + p}) \hookrightarrow \textup{Rist}_G(\widehat{k})$.

For the group $G_d$, the maps $\rho_{k + p, k}$ of the inverse system can be evaluated using the following remark. 

\begin{remark}
	Note that, by Remark \ref{length_twice}, for $g \in \textup{St}(\widehat{k+1})$ such that $\psi_{k+1}(g) = (a_1^{n_1}, a_1^{n_2}, \ldots, a_1^{n_{d^{k + 1}}})$, we have $\psi_{k}(g) = (h_1, \ldots, h_{d^{k}})$, for some $h_j \in G_d$ with $|h_j|_A = \frac{\sum_{i = 1}^{d}n_{(j-1)d + i}}{2}$. Then the coset $(h_1, \ldots, h_{d^{k}}){\textup{Rist}(\widehat{k})} = (a_1^{m_1}, a_1^{m_2}, \ldots, a_1^{m_{d^{k}}}){\textup{Rist}(\widehat{k})}$, where $m_i \equiv |h_i|_A \; (\textup{mod}\ 2^{k+1})$ with $0 \le m_i < 2^{k+1}$. Hence, for the group $G_d$, the map $$\rho_{k + 1, k} : \textup{St}(\widehat{k + 1})/ \textup{Rist}(\widehat{k + 1}) \to \textup{St}(\widehat{k})/ \textup{Rist}(\widehat{k})$$ sends $$(a_1^{n_1}, a_1^{n_2}, \ldots, a_1^{n_{d^{k + 1}}}){\textup{Rist}(\widehat{k+1})} \mapsto (a_1^{m_1}, a_1^{m_2}, \ldots, a_1^{m_{d^{k}}}){\textup{Rist}(\widehat{k})}.$$
\end{remark}

\begin{theorem}\label{rig_ker}
	The rigid kernel of $G_d$ is the abelian group
	$$\prod_{k \in \mathds{N}} \mathcal{C}_{2^{k}}^{d^k-d^{k-1}},$$ where $\mathcal{C}_n$ is the cyclic group $\mathds{Z}/n\mathds{Z}$.
\end{theorem}

\begin{proof}
	In view of Equation (\ref{quot_elts}), note that the elements of $\textup{St}(\widehat{k})/\textup{Rist}(\widehat{k})$ are of the form $(a_1^{n_1}, \ldots, a_1^{n_{d^{k}}})\textup{Rist}(\widehat{k})$ with appropriate powers of $a_1$, i.e., for  $1 \le r\le k$ and $0 \le t < d^{k - r}$, $\sum_{i = 1}^{d^r}n_{d^rt + i} \equiv 0 \ (\textup{mod}\ 2^{r+1})$, where $0 \le n_{j} < 2^{k+1}$. Accordingly, as earlier, we denote an element $(a_1^{n_{k,1}}, \ldots, a_1^{n_{k, d^{k}}})\textup{Rist}(\widehat{k}) \in \textup{St}(\widehat{k})/\textup{Rist}(\widehat{k})$  by $\llbracket n_{k,j} \rrbracket_{1\le j \le d^k}$ or simply by $\llbracket n_{k,j} \rrbracket$, as the level parameter $k$ is  included as a subscript to the respective numbers. 
	Thus, as a subgroup of $\prod_{k \in \mathds{N}} \textup{St}(\widehat{k})/\textup{Rist}(\widehat{k})$, the rigid kernel of $G_d$ is 
	\begin{eqnarray}\label{rigid_set}
		\qquad \left\{\left(\llbracket n_{k,j} \rrbracket\right)_{k \in \mathds{N}} \in \prod_{k \in \mathds{N}} \textup{St}(\widehat{k})/\textup{Rist}(\widehat{k}) : n_{k, j} \equiv \frac{\sum_{i = 1}^{d}n_{k+1, jd - d + i}}{2}\ (\text{mod}\ 2^{k+1})
		\right\}.
	\end{eqnarray}
	By considering $r = 1$ in Equation (\ref{quot_elts}), observe that $\sum_{i = 1}^{d}n_{k+1, jd - d + i} \equiv 0\ (\text{mod}\ 4)$,  for $1 \le j \le d^k$. Hence, since $n_{k, j} \equiv \frac{\sum_{i = 1}^{d}n_{k+1, jd - d + i}}{2}\ (\text{mod}\ 2^{k+1})$ in the rigid kernel, we have $n_{k, j}$ is even for all $j$.
	
	For all $k \in \mathds{N}$ and $1 \le j \le d^k$ with $d \ndivides j$, given arbitrary even numbers $n_{k,j}$ in the range $0 \le n_{k,j} < 2^{k+1}$, we show that there exist unique numbers $n_{k,j}$ for $d \divides j$ such that $\left(\llbracket n_{k,j} \rrbracket\right)_{k \in \mathds{N}}$ in the rigid kernel, i.e., the numbers $n_{k, j}$'s satisfy congruences in equations (\ref{quot_elts}) and (\ref{rigid_set}). In fact, we construct $n_{1, d}$ using Equation (\ref{quot_elts}); then for $k \ge 2$, by induction, we construct $n_{k, j}$'s using Equation (\ref{rigid_set}) and show that $n_{k, j}$'s satisfy Equation (\ref{quot_elts}). 
	
	For $k = 1$, consider a number $n_{1, d}$ such that $n_{1,d} \equiv - \sum_{i=1}^{d-1}n_{1, i}\ (\text{mod}\ 4)$ and $0 \le n_{1,d} < 4$. Since $\sum_{i=1}^{d-1}n_{1, i}$ is even, observe that $n_{1,d}$ is even and unique satisfying $\sum_{i=1}^{d}n_{1, i}  \equiv 0\ (\text{mod}\ 4)$. 
	
	For $k \ge 1$, suppose there exists $n_{k, j}$ for $d \divides j$ and $1 \le j \le d^{k}$ as required for the rigid kernel. We now  show the existence of required numbers $n_{k+1, j}$ with $d \divides j$ for $(k+1)$-th level.  For $0 \le j \le d^{k+1}$ and $d \divides j$, consider $n_{k+1, j}$ in the range $0 \le n_{k+1, j} < 2^{k+2}$ as per the following: 
	\begin{equation}\label{h_ij}
		n_{k+1, j} \equiv 2n_{k, \frac{j}{d}}-\sum_{i=1}^{d-1}n_{k+1, j-d+i}\ (\text{mod}\ 2^{k+2})
	\end{equation}
	Since $d \divides j$, the numbers $n_{k+1, j-d+i}$ (for $1 \le i \le d-1$) on the right hand side of Equation (\ref{h_ij}) are available by the hypothesis as $d \ndivides j-d+i$. Since $\sum_{i=1}^{d-1}n_{k+1, j-d+i}$ is even, note that $n_{k+1, j}$ is even. Further, since $0 \le n_{k+1, j} < 2^{k+2}$, as per construction in Equation (\ref{h_ij}), $n_{k+1, j}$ is unique satisfying $n_{k, j} \equiv \frac{\sum_{i = 1}^{d}n_{k+1, jd - d + i}}{2}\ (\text{mod}\ 2^{k+1})$, as desired. We now show that $n_{k,j}$'s also satisfy Equation (\ref{quot_elts}) inductively. For $k \ge 1$, suppose $\sum_{s = 1}^{d^r}n_{k, d^rt + s} \equiv 0 \ (\textup{mod}\ 2^{r+1})$, for $1\le r\le k$ and $0 \le t < d^{k - r}$. Using Equation (\ref{h_ij}), substitute the values of $n_{k, d^rt + s}$ to get 
	$$\sum_{s = 1}^{d^r}\frac{\sum_{i=1}^{d}n_{k+1,(d^rt + s)d-d+i}}{2} \equiv 0 \ (\textup{mod}\ 2^{r+1}).$$ Hence, $\sum_{s = 1}^{d^{r+1}}n_{k+1,d^{r+1}t + s} \equiv 0 \ (\textup{mod}\ 2^{r+2}).$ Now set $r+1 = r'$ so that, for $2 \le r' \le k+1$ and $0 \le t < d^{k+1-r'}$, we have $\sum_{s = 1}^{d^{r'}}n_{k+1, d^{r'}t + s} \equiv 0 \ (\textup{mod}\ 2^{r'+1})$. Also, note that $\sum_{s=1}^{d}n_{k+1,j-d+s} \equiv 0\ (\text{mod}\ 4)$ because, by Equation (\ref{h_ij}), for $1 \le j \le d^{k+1}$ such that $d \divides j$, we have $ \sum_{s=1}^{d}n_{k+1,j-d+s} \equiv 2n_{k, \frac{j}{d}}\ (\text{mod}\ 2^{k+2})$ and $n_{k, \frac{j}{d}}$ is even. Therefore, for $1\le r' \le k+1$ and $0 \le t < d^{k+1-r'}$, we have $\sum_{s = 1}^{d^{r'}}n_{k+1, d^{r'}t + s} \equiv 0 \ (\textup{mod}\ 2^{r'+1})$. Hence, 
	$(\llbracket n_{k,j} \rrbracket)_{k \in \mathds{N}}$ is in the rigid kernel $\underset{k \ge 1}{\underleftarrow{\text{lim}}}\ \textup{St}(\widehat{k})/ \textup{Rist}(\widehat{k})$.

	Now consider the map 
		\begin{align*}
			\phi: \underset{k \ge 1}{\underleftarrow{\text{lim}}}\ \textup{St}(\widehat{k})/ \textup{Rist}(\widehat{k}) \to \prod_{k \ge 1} \mathcal{C}_{2^{k}}^{d^k-d^{k-1}},  
		\end{align*}  
	which sends $h = (\llbracket n_{1,j} \rrbracket, \llbracket n_{2,j} \rrbracket, \ldots)$ in the rigid kernel to $\eta = (\eta_1, \eta_2, \ldots)$, where $\eta_k$ is the $(d^k-d^{k-1})$-tuple  $\left(\frac{n_{k, j}}{2}\right)_{\underset{d \ndivides j}{1 \le j \le d^k}}$. Note that $\eta_k \in \mathcal{C}_{2^{k}}^{d^k-d^{k-1}}$. To show $\phi$ is a homomorphism, for $g_k = \llbracket m_{k,j} \rrbracket$ and $h_k = \llbracket n_{k,j} \rrbracket$, let $g = (g_k)_{k \in \mathds{N}}$ and $h = (h_k)_{k \in \mathds{N}}$ be in the rigid kernel. Note that $g_kh_k = \left(a_1^{p_{k, 1}}, \ldots, a_1^{p_{k, d^k}}\right)\textup{Rist}(\widehat{k}) = \llbracket p_{k,j} \rrbracket$, where $0 \le p_{k, j}(\equiv m_{k,j}+n_{k,j} (\text{mod}\ 2^{k+1})) < 2^{k+1}$. Now 
	\begin{eqnarray*}
		\phi(gh) &=& \phi \big(\left(g_kh_k\right)_{k \in \mathds{N}}\big)\\
		&=& \phi \big(\left(\llbracket p_{k,j} \rrbracket\right)_{k \in \mathds{N}}\big)\\
		&=& \left(\Big(\frac{p_{k, j}}{2}\Big)_{\underset{d \ndivides j}{1 \le j \le d^k}}\right)_{k \in \mathds{N}}\\
		&=& \left(\left(\frac{m_{k,j}}{2}\right)_{\underset{d \ndivides j}{1 \le j \le d^k}}\right)_{k \in \mathds{N}} + \left( \left(\frac{n_{k,j}}{2}\right)_{\underset{d \ndivides j}{1 \le j \le d^k}}\right)_{k \in \mathds{N}}\\
		& = & \phi(g) + \phi(h)
	\end{eqnarray*}
	
	To show that $\phi$ is onto, let $\eta = \left((\eta_{k, j})_{\underset{d \ndivides j}{1 \le j \le d^k}}\right)_{k \in \mathds{N}} \in \prod_{k \in \mathds{N}} \mathcal{C}_{2^{k}}^{d^k-d^{k-1}}$ where $\eta_{k, j} \in \mathcal{C}_{2^{k}}$. Then, for $k \in \mathds{N}$ and $1 \le j \le d^{k}$ with $d\ndivides j$, set $n_{k, j} = 2\eta_{k, j}$. Hence, for these values of $n_{k,j}$ with $d\ndivides j$, we get unique values for $n_{k,j}$ with $d \divides j$ and $h = (\llbracket n_{k,j} \rrbracket)_{k \in \mathds{N}} \in \underset{k \ge 1}{\underleftarrow{\text{lim}}}\ \textup{St}(\widehat{k})/ \textup{Rist}(\widehat{k})$ such that $\phi(h) = \eta$. 
	
	To show that $\phi$ is one-one, consider $g = (\llbracket m_{k,j} \rrbracket)_{k \in \mathds{N}}$ and $h = (\llbracket n_{k,j} \rrbracket)_{k \in \mathds{N}}$ in the rigid kernel such that $\phi(g) = \phi(h)$. Accordingly, for $d \ndivides j$, we have $\frac{m_{k, j}}{2} = \frac{n_{k, j}}{2}$. Thus, for $d \divides j$, $m_{k,j} = n_{k,j}$, as these values exist uniquely corresponding to the given values when $d \ndivides j$. Hence, $g = h$.
\end{proof}

It is evident from the structure of rigid kernel of $G_d$ (given Theorem \ref{rig_ker}) that it has both finite order and infinite order elements. Hence, we have the following corollary.
	
\begin{corollary}
	The rigid kernel of the group $G_d$ is neither torsion nor torsion-free. 
\end{corollary} 	

\begin{proposition}
	In $G_d$, the branch kernel is non-trivial.
\end{proposition}

\begin{proof}
	Consider the subgroup $H = \{w \in G_d : |w|_A \equiv 0\ (\textup{mod}\ 4)\}$. Clearly $G_d/ H = \{H, a_1H, a_1^2H, a_1^3H\}$. For each $k \in \mathds{N}$, there exists $g_k \in \textup{Rist}(\widehat{k})$ such that $\psi_k(g_k) = (a_1^{2^{k+1}}, e, \ldots, e)$. By Remark \ref{length_twice}, observe that $|g_k|_A = 2$, which implies $g_k \notin H$. Hence, for any $k$, $\textup{Rist}(\widehat{k})$ is not a subgroup of $H$.
\end{proof}

The notion of the Hausdorff dimension was initially defined as a measure of fractalness of sets over the reals and then it is generalized to metric spaces. In \cite{Barnea_1997} Barnea and Shalev adopted the concept of  Hausdorff dimension to metric spaces arising in profinite groups. Since $\text{Aut}(T)$ is a profinite group, for a subgroup $G$ of $\text{Aut}(T)$, the Hausdorff dimension $\dim_H(\overline{G})$ of the closure of $G$, as given by the following formula, was computed for various groups (e.g., \cite{Alcober_2014}):  $$\dim_H(\overline{G}) = \liminf_{k \to \infty} \frac{\log \ [G : {\textup{St}_{G}(\widehat{k})}]}{\log \ [\textup{Aut}(T) : {\textup{St}_{\textup{Aut}(T)}(\widehat{k})}]}.$$

In the following result, we determine the index of a level stabilizer in $G_d$; using which, we compute the Hausdorff dimension of $G_d$ in Theorem \ref{haus}.

\begin{theorem}
	For $k \ge 1$, the index of ${\textup{St}(\widehat{k+1})}$ in ${\textup{St}(\widehat{k})}$ and hence the index of ${\textup{St}(\widehat{k+1})}$ in $G_d$ are given by \[\left[{\textup{St}(\widehat{k})} : {\textup{St}(\widehat{k+1})}\right] = \frac{{d!}^{d^{k}}}{2^{d^{k-1}}} \qquad \text{and}\qquad \left[G_d : {\textup{St}(\widehat{k+1})}\right] = \frac{{d!}^{\frac{d^{k+1}-1}{d-1}}}{2^{\frac{d^{k}-1}{d-1}}}.\]
\end{theorem}

\begin{proof}
	Here, for an element $g \in \text{Aut}(T)$, we write $\lambda_{g}^{(k)}$ to represent the induced action of $g$ on $X^k$ so that $\lambda_g^{(1)} = \lambda_g$. 
	Let $H = \textup{St}(\widehat{k+1})$. For $g, h \in \textup{St}(\widehat{k})$, we have $$gH = hH \Leftrightarrow g^{-1}h \in H \Leftrightarrow \lambda_{g^{-1}h}^{(k+1)} = 1  \Leftrightarrow \lambda_{g}^{(k+1)} = \lambda_{h}^{(k+1)} \Leftrightarrow \lambda_{g}^{(k+1)}(ux) = \lambda_{h}^{(k+1)}(ux),$$ for any $u \in X^{k}$ and $x \in X$, i.e., $u\lambda_{g|_{u}}^{(1)}(x) = u\lambda_{h|_{u}}^{(1)}(x) \Leftrightarrow \lambda_{g|_{u}}(x) = \lambda_{h|_{u}}(x).$ 
	Hence, for $g, h \in \textup{St}(\widehat{k})$, we have $gH \neq hH$ if and only if $\lambda_{g|_{u}} \neq \lambda_{h|_{u}}$, for some $u \in X^{k}$. Accordingly, the index $[{\textup{St}(\widehat{k})} : {\textup{St}(\widehat{k+1})}]$ equals 
	$$|\{(\lambda_{g_1}, \ldots, \lambda_{g_{d^{k}}}) : \psi_{k}(g) = (g_1, \ldots, g_{d^{k}}), \text{ for } g \in \textup{St}(\widehat{k}) \}|.$$
	 For $g \in \textup{St}(\widehat{k})$, if $\psi_{k}(g) = (g_1, \ldots, g_{d^{k}})$, then we have $\sum_{i = 1}^{d}|g_{rd + i}|_A \equiv 0(\text{mod }4)$ for all $0 \le r  < d^{k-1}$, by Theorem \ref{stab}. Clearly, for each $0 \le r < d^{k-1}$,  $|g_{rd + i}|_A$ is even for odd number of $i$'s in $\{1, \ldots, d\}$ and $|g_{rd + i}|_A$ is odd for remaining even number of $i$'s (as $d$ is odd). For any $h \in G_d$, since $|h|_A$ is even if and only if $\lambda_{h}$ is an even permutation, we have the following: for each $0 \le r  < d^{k-1}$, $\lambda_{g_{rd + i}}$ is an even permutation for odd number of $i$'s in $\{1, \ldots, d\}$ and $\lambda_{g_{rd + i}}$ is an odd permutation for even number of $i$'s.
	 
We claim that, for any odd integer $n \in \mathds{Z}$ and every odd permutation $\lambda$ in $S_d$, there exists a $g \in G_d$ such that $\lambda_g = \lambda$ and $|g|_A = n$. First note that, by Remark \ref{odd_even}, there exists $h \in G_d$ such that $\lambda_h = \lambda$ and $|h|_A$ is odd, say $m$. Now consider $g = ha_1^{n-m} \in G_d$ and observe that $|g|_A = n$. Further, since $n-m$ is even and $\lambda_{a_1^2} = 1$, we have $\lambda_g = \lambda_h\lambda_{a_1^{n-m}} = \lambda_h = \lambda$. Similarly, for any even integer $n \in \mathds{Z}$ and every even permutation $\lambda$ in $S_d$, there exists a $g \in G_d$ such that $\lambda_g = \lambda$ and $|g|_A = n$.  

  Consider the set $\mathcal{B}$ of $d^{k}$-tuples $(\lambda_1,  \ldots, \lambda_{d^{k}})$ consisting of permutations in $S_d$ such that, for $0 \le i < d^{k-1}$, in every block $B_i = \{\lambda_{di + 1}, \ldots, \lambda_{di+d}\}$ of $d$ permutations, the number of odd permutations is even. Let $(\lambda_1,  \ldots, \lambda_{d^{k}}) \in \mathcal{B}$ with $2t_i$ number of odd permutations in the corresponding $B_i$. We show that there exists $g \in \textup{St}(\widehat{k})$ with $\psi_{k}(g) = (g_1, \ldots, g_{d^{k}})$ such that $(\lambda_{g_1}, \ldots, \lambda_{g_{d^{k}}}) = (\lambda_1,  \ldots, \lambda_{d^{k}})$. Corresponding to odd permutations in $B_i$, consider the integers $1$ and $-1$ for $t_i$ times each; and the integer $0$ for each of the remaining even permutations.  For each of these integers and the corresponding odd or even permutation $\lambda_i$, by the previous paragraph, there exists $g_i \in G_d$ such that $\lambda_{g_i} = \lambda_i$ and $|g_i|_A$ is the corresponding integer $1$ or $-1$ or $0$. Note that, for each $0 \le i < d^{k-1}$, $\sum_{j=1}^d |g_{di + j}|_A = t_i \times 1 + t_i \times (-1) + (d-2t_i) \times 0 = 0$. Therefore, for $1 \le r \le k$ and $0 \le t < d^{k - r}$, 
  $$\sum_{l = 1}^{d^r}|g_{d^rt + l}|_A  = 0.$$
	Thus, by Theorem \ref{stab}, there is $g \in \textup{St}(\widehat{k})$ with $\psi_{k}(g) = (g_1, \ldots, g_{d^{k}})$, as desired. Hence, the index $[{\textup{St}(\widehat{k})} : {\textup{St}(\widehat{k+1})}] = |\mathcal{B}|$. We now count the number of elements in $\mathcal{B}$ as per the following. Note that the number of blocks $B_i$'s having odd permutations is $\sum_{l = 0}^{d^{k-1}}\binom{d^{k-1}}{l}$. Corresponding to each such combination of $l$ blocks, the number of positions consisting of odd permutations is ${\Big(\sum_{j = 1}^{\frac{d-1}{2}}\binom{d}{2j}\Big)}^l$. Further, at each position in $d^k$-tuple we have $d!/2$ possibilities of even or odd permutations. Hence,
    \begin{eqnarray*}
    	\left[{\textup{St}(\widehat{k})} : {\textup{St}(\widehat{k+1})}\right] &=& {\biggl(\frac{d!}{2}\biggl)}^{d^{k}}\sum_{l = 0}^{d^{k-1}}\binom{d^{k-1}}{l}{\biggl[\sum_{j = 1}^{\frac{d-1}{2}}\binom{d}{2j}\biggl]}^l\\
    	&=& {\biggl(\frac{d!}{2}\biggl)}^{d^{k}}\sum_{l = 0}^{d^{k-1}}\binom{d^{k-1}}{l}{(2^{d-1}-1)}^l\\
    	&=& {\biggl(\frac{d!}{2}\biggl)}^{d^{k}}{(1 + (2^{d-1}-1))}^{d^{k-1}}\\
    	&=& \frac{{d!}^{d^{k}}}{2^{d^{k-1}}}.
    \end{eqnarray*}
    Also, since $S_d$ is generated by $\lambda_{a_i}$'s, $[G_d : {\textup{St}(\widehat{1})}] = d!$. Hence, 
    \begin{eqnarray*}
    	[G_d : {\textup{St}(\widehat{k+1})}] &=& [G_d : {\textup{St}(\widehat{1})}] [{\textup{St}(\widehat{1})} : {\textup{St}(\widehat{2})}] \cdots [{\textup{St}(\widehat{k})} : {\textup{St}(\widehat{k+1})}]\\
    	&=& d! \ \frac{{d!}^{d}}{2} \frac{{d!}^{d^{2}}}{2^{d}}\frac{{d!}^{d^{3}}}{2^{d^{2}}}\cdots \frac{{d!}^{d^{k}}}{2^{d^{k-1}}}\\
    	&=& \frac{{d!}^{\frac{d^{k+1}-1}{d-1}}}{2^{\frac{d^{k}-1}{d-1}}}.
    \end{eqnarray*}
\end{proof}

\begin{theorem}\label{haus}
	$\dim_H(\overline{G_d}) = 1 - \dfrac{\log2}{d\log d!}$.
\end{theorem}

\begin{proof}
	Note that
	\begin{eqnarray*}
		[\textup{Aut}(T) : \textup{St}_{\textup{Aut}(T)}(\widehat{k})]  & = &  |\underset{k\ \text{times}}{S_d \wr S_d \wr \cdots \wr S_d}|\\
	& = &  d!^{\frac{d^k-1}{d-1}}
	\end{eqnarray*}
	
    Hence,
    \begin{eqnarray*}
    	\dim_H(\overline{G_d}) 
    	& = & \displaystyle\liminf_{k \to \infty} \frac{\log \frac{{d!}^{\frac{d^{k}-1}{d-1}}}{2^{\frac{d^{k-1}-1}{d-1}}}}{\log d!^{\frac{d^k-1}{d-1}}}\\
    	& = & \liminf_{k \to \infty} \left(1 - \frac{d^{k-1}-1}{d^k - 1}\frac{\log 2}{\log d!}\right)\\
    	& = & 1 - \frac{\log2}{d\log d!}
    \end{eqnarray*}	
\end{proof}


\begin{thebibliography}{10}
	
	\bibitem{Barnea_1997}
	Y.~Barnea and A.~Shalev.
	\newblock Hausdorff dimension, pro-{$p$} groups, and {K}ac-{M}oody algebras.
	\newblock {\em Trans. Amer. Math. Soc.}, 349(12):5073--5091, 1997.
	
	\bibitem{Bartholdi2003presentation}
	L.~Bartholdi.
	\newblock Endomorphic presentations of branch groups.
	\newblock {\em J. Algebra}, 268(2):419--443, 2003.
	
	\bibitem{Bartholdi2003}
	L.~Bartholdi, R.~I. Grigorchuk, and Z.~\v{S}uni\'{k}.
	\newblock Branch groups.
	\newblock In {\em Handbook of algebra, {V}ol. 3}, pages 989--1112.
	Elsevier/North-Holland, Amsterdam, 2003.
	
	\bibitem{Bartholdi_2012}
	L.~Bartholdi, O.~Siegenthaler, and P.~Zalesskii.
	\newblock The congruence subgroup problem for branch groups.
	\newblock {\em Israel J. Math.}, 187:419--450, 2012.
	
	\bibitem{bartholdi2001word}
	L.~Bartholdi and Z.~\v{S}uni\'{k}.
	\newblock On the word and period growth of some groups of tree automorphisms.
	\newblock {\em Comm. Algebra}, 29(11):4923--4964, 2001.
	
	\bibitem{Bass_1964}
	H.~Bass, M.~Lazard, and J.-P. Serre.
	\newblock Sous-groupes d'indice fini dans {${\bf SL}(n,\,{\bf Z})$}.
	\newblock {\em Bull. Amer. Math. Soc.}, 70:385--392, 1964.
	
	\bibitem{Bou-Rabee2020}
	K.~Bou-Rabee, R.~Skipper, and D.~Studenmund.
	\newblock Commensurability growth of branch groups.
	\newblock {\em Pacific J. Math.}, 304(1):43--54, 2020.
	
	\bibitem{Dahmani2005}
	F.~Dahmani.
	\newblock An example of non-contracting weakly branch automaton group.
	\newblock In {\em Geometric methods in group theory}, volume 372 of {\em
		Contemp. Math.}, pages 219--224. Amer. Math. Soc., 2005.
	
	\bibitem{Domenico_2022_basilica}
	E.~Di~Domenico, G.~A. Fern\'andez-Alcober, M.~Noce, and A.~Thillaisundaram.
	\newblock {$p$}-{B}asilica groups.
	\newblock {\em Mediterr. J. Math.}, 19(6):Paper No. 275, 28, 2022.
	
	\bibitem{Garrido_2017}
	G.~A. Fern\'andez-Alcober, A.~Garrido, and J.~Uria-Albizuri.
	\newblock On the congruence subgroup property for {GGS}-groups.
	\newblock {\em Proc. Amer. Math. Soc.}, 145(8):3311--3322, 2017.
	
	\bibitem{Noce2020engel}
	G.~A. Fern\'{a}ndez-Alcober, M.~Noce, and G.~M. Tracey.
	\newblock Engel elements in weakly branch groups.
	\newblock {\em J. Algebra}, 554:54--77, 2020.
	
	\bibitem{Alcober_2014}
	G.~A. Fern\'andez-Alcober and A.~Zugadi-Reizabal.
	\newblock G{GS}-groups: order of congruence quotients and {H}ausdorff
	dimension.
	\newblock {\em Trans. Amer. Math. Soc.}, 366(4):1993--2017, 2014.
	
	\bibitem{Francoeur2020}
	D.~Francoeur.
	\newblock On maximal subgroups of infinite index in branch and weakly branch
	groups.
	\newblock {\em J. Algebra}, 560:818--851, 2020.
	
	\bibitem{francoeur2024structure}
	D.~Francoeur, R.~Grigorchuk, P.-H. Leemann, and T.~Nagnibeda.
	\newblock On the structure of finitely generated subgroups of branch groups.
	\newblock {\em arXiv preprint arXiv:2402.15496}, 2024.
	
	\bibitem{Garrido_2016_gupta}
	A.~Garrido.
	\newblock Abstract commensurability and the {G}upta-{S}idki group.
	\newblock {\em Groups Geom. Dyn.}, 10(2):523--543, 2016.
	
	\bibitem{Garrido_2016}
	A.~Garrido.
	\newblock On the congruence subgroup problem for branch groups.
	\newblock {\em Israel J. Math.}, 216(1):1--13, 2016.
	
	\bibitem{rigid2023branch}
	A.~Garrido and Z.~{\v{S}}uni{\'c}.
	\newblock Branch groups with infinite rigid kernel.
	\newblock {\em arXiv preprint arXiv:2310.06581}, 2023.
	
	\bibitem{Garrido_2019}
	A.~Garrido and J.~Uria-Albizuri.
	\newblock Pro-{$\mathcal{C}$} congruence properties for groups of rooted tree
	automorphisms.
	\newblock {\em Arch. Math. (Basel)}, 112(2):123--137, 2019.
	
	\bibitem{Grigorchuk2005}
	R.~Grigorchuk and Z.~\v{S}uni\'{c}.
	\newblock Self-similarity and branching in group theory.
	\newblock In {\em Groups {S}t. {A}ndrews 2005. {V}ol. 1}, volume 339 of {\em
		London Math. Soc. Lecture Note Ser.}, pages 36--95. Cambridge Univ. Press,
	Cambridge, 2007.
	
	\bibitem{Grigorchuk2006hanoi}
	R.~Grigorchuk and Z.~\v{S}uni\'{k}.
	\newblock Asymptotic aspects of {S}chreier graphs and {H}anoi {T}owers groups.
	\newblock {\em C. R. Math. Acad. Sci. Paris}, 342(8):545--550, 2006.
	
	\bibitem{Grigorchuk2000}
	R.~I. Grigorchuk.
	\newblock {\em Just Infinite Branch Groups}, pages 121--179.
	\newblock Birkh{\"a}user Boston, Boston, MA, 2000.
	
	\bibitem{grigorchuk2002torsion}
	R.~I. Grigorchuk and A.~\.{Z}uk.
	\newblock On a torsion-free weakly branch group defined by a three state
	automaton.
	\newblock {\em Internat. J. Algebra Comput.}, 12(1-2):223--246, 2002.
	
	\bibitem{grigorchuk1980burnside}
	R.~I. Grigor\v{c}uk.
	\newblock On {B}urnside's problem on periodic groups.
	\newblock {\em Funktsional. Anal. i Prilozhen.}, 14(1):53--54, 1980.
	
	\bibitem{gupta1983burnside}
	N.~Gupta and S.~Sidki.
	\newblock On the {B}urnside problem for periodic groups.
	\newblock {\em Math. Z.}, 182(3):385--388, 1983.
	
	\bibitem{mamaghani2011fractal}
	M.~J. Mamaghani.
	\newblock A fractal non-contracting class of automata groups.
	\newblock {\em Bull. Iranian Math. Soc.}, 29(2):51--64, 2003.
	
	\bibitem{Nekrashevych2005}
	V.~Nekrashevych.
	\newblock {\em Self-similar groups}, volume 117 of {\em Mathematical Surveys
		and Monographs}.
	\newblock American Mathematical Society, Providence, RI, 2005.
	
	\bibitem{Noce2021}
	M.~Noce.
	\newblock A family of fractal non-contracting weakly branch groups.
	\newblock {\em Ars Math. Contemp.}, 20(1):29--36, 2021.
	
	\bibitem{Noce2021Hausdorff}
	M.~Noce and A.~Thillaisundaram.
	\newblock Hausdorff dimension of the second {G}rigorchuk group.
	\newblock {\em Internat. J. Algebra Comput.}, 31(6):1037--1047, 2021.
	
	\bibitem{Pervova_2007}
	E.~Pervova.
	\newblock Profinite completions of some groups acting on trees.
	\newblock {\em J. Algebra}, 310(2):858--879, 2007.
	
	\bibitem{saha2023branch}
	S.~Saha and K.~V. Krishna.
	\newblock A branch group in a class of non-contracting weakly regular branch
	groups.
	\newblock {\em Internat. J. Algebra Comput.}, 34(06):901--918, 2024.
	
	\bibitem{Thesis_Skipper}
	R.~Skipper.
	\newblock {\em On a generalization of the Hanoi Towers group}.
	\newblock PhD thesis, Binghamton University, 2018.
	
	\bibitem{Skipper2020}
	R.~Skipper.
	\newblock The congruence subgroup problem for a family of branch groups.
	\newblock {\em Internat. J. Algebra Comput.}, 30(2):397--418, 2020.
	
\end{thebibliography}
\end{document}